\documentclass[12pt]{amsart}

\usepackage[usenames,dvipsnames]{xcolor}

\usepackage{amsmath,amsthm,amsfonts,graphicx,amssymb,amscd}
\usepackage[all,cmtip]{xy}
\usepackage{graphicx, overpic, float}
\usepackage[mathscr]{euscript}
\usepackage{hyperref}
\usepackage{pxfonts}

\usepackage{enumerate}

\newtheorem{thm}{Theorem}

\newtheorem{cor}[thm]{Corollary}
\newtheorem{lem}[thm]{Lemma}

\newtheorem{prop}[thm]{Proposition}

\theoremstyle{definition}

\newcommand{\C}{\mathbb{C}}
\newcommand{\D}{\mathbb{D}}
\newcommand{\R}{\mathbb{R}}
\newcommand{\slr}{\mathrm{SL}_2(\mathbb{R})}
\newcommand{\pslc}{\mathrm{PSL}_2(\mathbb{C})}
\newcommand{\Te}{Teichm\"{u}ller }
\newcommand{\tg}{\mathcal{T}_g }
\newcommand{\mg}{\mathcal{M}_g }
\newcommand{\cp}{\mathbb{C}\mathrm{P}^1}

\usepackage{color}

\makeindex

\title[Quadratic differentials in Teichm\"{u}ller theory]{Holomorphic quadratic differentials\\ in Teichm\"{u}ller theory}

\author{Subhojoy Gupta}

\address{Department of Mathematics, Indian Institute of Science, Bangalore 560012, India.} 

\email{subhojoy@iisc.ac.in}

\begin{document}

\begin{abstract} This expository survey describes how holomorphic quadratic differentials arise in several aspects of \Te theory,  highlighting their relation with various geometric structures on surfaces. The final section summarizes results for non-compact surfaces of finite type, when the quadratic differential has poles of finite order at the punctures. 
\end{abstract}

\setcounter{tocdepth}{1}

\maketitle
\tableofcontents

\section{Introduction}

Let $S$ be a closed oriented smooth surface of genus $g\geq 2$. 
The \Te space $\mathcal{T}_g$ of $S$ is the space of marked complex structures on $S$, as well as marked hyperbolic structures on $S$.
These two equivalent viewpoints give rise to a rich interaction between the complex-analytic and geometric tools in \Te theory. 
This survey  aims to highlight some of this interplay, in the context of holomorphic quadratic differentials, which 
 have been a vital part of \Te theory since its inception (see, for example, \cite{Teich} and the commentary \cite{Teich-comm}).  \\

The selection of topics here is guided by our own interests, and does not purport to be comprehensive; the purpose is to provide a glimpse into several different aspects of \Te theory that involve quadratic differentials in an essential way.  We have aimed to keep the exposition light, to make this survey accessible to a broad audience; in particular, we shall often refer to other sources for proofs, or a more detailed treatment. 
Indeed, there is a vast literature on these topics, and we provide an extensive, but necessarily incomplete, bibliography  at the end of this article.  For some standard books on the subject, see \cite{Streb}, \cite{Gardiner}, \cite{Hubbard}, just to name a few. \\

In the final section, we focus on some recent work concerning  holomorphic quadratic differentials on \textit{punctured} Riemann surfaces, with higher order poles at the punctures. There, we describe generalizations of some of the results for a closed surface  discussed in the preceding sections. One of the novel features of this non-compact setting  is that one needs to define structures on the non-compact ends that depend on some additional parameters at each pole. \\

In a broader context, holomorphic $k$-differentials where $k\geq 3$ arise in higher \Te theory  -- see \cite{Wienhard-et-al}, and the conclusion of \S5\index{higher Teichm\"{u}ller theory}. In particular, holomorphic cubic differentials are known to parametrize the space of convex projective structures on a surface (see \cite{Labourie, Loftin0}) and there has been recent work that develops the correspondence for meromorphic cubic differentials as well (see \cite{Loftin2, Dumas-Wolf, Nie}).  It is an active field of research to develop tools and theorems that generalize the results in the quadratic ($k=2$) case, that this survey discusses, to the case of higher order differentials. \\

\medskip

\textbf{Acknowledgments.} I wish to thank Athanase Papadopoulos for his kind invitation to write this article.  I am grateful to the SERB, DST (Grant no. MT/2017/000706) and the Infosys Foundation for their support.  It is a pleasure to thank Michael Wolf, for sharing his insight, and the many hours of conversation that led to some of the results described in \S8. Last but not the least, I am grateful to Fred Gardiner, and the anonymous referees,  for their suggestions that helped improve this article.

\section{The co-tangent space of \Te space}

Let $X$ be a Riemann surface of genus $g\geq 2$.  We shall assume $X$ is marked,  where a \textit{marking} is a choice of a homotopy class of a diffeomorphism from a fixed smooth surface $S$ to $X$.  Throughout, two such marked surfaces equipped with a structure (\textit{e.g.} a complex structure or hyperbolic metric) are \textit{equivalent} if they are isomorphic via a  map  that preserves the markings.

We formally introduce the central objects of this survey:

A \textit{holomorphic quadratic differential} on $X$ is a holomorphic section of $K \otimes K$, where $K$ is the canonical line bundle on $X$. Locally, such a holomorphic tensor has the form $q(z)dz\otimes dz$ (often written $q(z)dz^2$) where $q(z)$ is a holomorphic function.

Throughout this article, let $Q(X)$ be the complex vector space of holomorphic quadratic differentials  on $X$. Since the canonical line bundle $K$ has degree $2g-2$, the Riemann-Roch formula\index{Riemann-Roch} implies $\text{dim} Q(X) =  \text{deg}(K^2) - g+1 = 3g-3$.

For details of this computation, refer to \cite{FarkasKra} or \cite{Jost}, or any standard reference for Riemann surfaces. \\

The holomorphic quadratic differentials defined above immediately arise in  \Te theory as objects that parametrize the space of infinitesimal deformations of hyperbolic or complex structures on a fixed surface $S$:

\begin{thm} Let $X \in \mathcal{T}_g$. Then the cotangent space $T_X^\ast \mathcal{T}_g$\index{Teichm\"{u}ller space!cotangent space} can be identified with the space $Q(X)$  of holomorphic quadratic differentials on $X$.
\end{thm}

\medskip

We sketch three proofs that arise from different ways of defining \Te space.

\begin{itemize}

\item A pair of conformal structures in $\tg$ are related by a  \textit{quasiconformal map} $f$ between them that determines a  \textit{Beltrami differential}  $\mu = f_{\bar{z}}/f_z$  determining its {dilatation} -- see the beginning of  \S4. Thus, Beltrami differential is a $(-1,1)$-differential: it is is  locally expressed as $\mu(z) \frac{d\bar{z}}{dz}$ where $\mu(z)$ is a measurable function having sup-norm less than $1$. 
Any variation of conformal structures is then a derivative 
\begin{equation}\label{dotmu}
\frac{d\mu_t}{dt}\vert_{t=0} = \dot{\mu}
\end{equation}
 that lies in the Banach space $L^\infty_{(-1,1)}(X)$. 

Teichm\"{u}ller's Lemma\index{Teichm\"{u}ller!lemma}  (see \S3.1 of \cite{Nag} ) then asserts that the subspace of Beltrami differentials\index{differential!Beltrami} that determine a \textit{trivial} variation is given by 
\begin{equation}\label{triv}
\mathcal{N} = \{ \mu \text{ }  \lvert \text{ } \langle \mu, \phi\rangle :=  \int_X \mu \phi =0  \text{ for every } \phi \in Q(X) \}
\end{equation}
and thus the \textit{ tangent}  space $T_X \mathcal{T}_g$ can be identified with the quotient space $L^\infty_{(-1,1)}(X)/\mathcal{N}$.

The non-degenerate pairing $\langle\cdot, \cdot \rangle$ in the definition above then identifies the cotangent space $T_X^\ast \mathcal{T}_g$ with $Q(X)$. For more details, see \S7 of \cite{Earle-Eells}, and see \cite{Ahlfors} for a more comprehensive introduction to \Te theory from this point of view.\\

\textit{Remark.} The \Te space $\tg$ can be seen to be homeomorphic to $\R^{6g-6}$ via the Fenchel-Nielsen\index{Fenchel-Nielsen} parametrization that considers length and twist parameters for the  $3g-3$ pants curves of a choice of a pants decomposition  on the genus-$g$ surface (see \cite{FarbMarg}). One way to find a basis of $T^\ast_X \mathcal{T}_g$  is to consider the quadratic differentials that are dual to length and twist deformations --  see \cite{Wolpert} or \cite{Masur-Ext}, or \cite{Rupflin} for some recent related work. On hyper-elliptic surfaces, it is often possible to also give a fairly explicit set of bases (\textit{c.f.} \cite{Veech}).  \\

\item From a differential geometric standpoint, the infinite-dimensional space $\mathcal{H}$ of hyperbolic metrics on the (smooth) surface $S$ has a tangent space at $g\in \mathcal{H}$ that has an orthogonal decomposition
\begin{equation*}
\hspace{.5in}T_g\mathcal{H} = \mathcal{S}_{(0,2)}(g)   \oplus \{\text{Lie derivatives of }g\text{ along smooth vector fields on }S\}
\end{equation*}
where $\mathcal{S}_{(0,2)}(g)$ is the space of \text{traceless and divergence-free symmetric } $(0,2)$-\text{tensors} on the hyperbolic surface $(S,g)$.
By quotienting by the action of the self-diffeomorphisms of $S$, we see that the tangent space to \Te space $\tg$ at $(S,g)$ is $\mathcal{S}_{(0,2)}(g)$. A brief computation (see pg. 46, in \S2.4 of \cite{Tromba}) then shows that any such symmetric $(0,2)$-tensor is in fact the real part of a holomorphic quadratic differential (and vice versa), and we again obtain an identification with $Q(X)$.\\

\item Yet another approach comes from algebraic geometry: \Te space is the universal cover of the moduli space of algebraic curves $\mathcal{M}_g$. By the Kodaira-Spencer deformation theory, the tangent space $T_X\tg$ is canonically identified with $H^1(X, K^{-1})$ where the dual of the canonical line bundle $K^{-1}$ is the sheaf of holomorphic vector fields on $X$. By Serre duality,  this space is dual to $H^0(X,K^2)$, and so the \textit{co-}tangent space is identified with $Q(X)$.  For more details, see the concluding remarks of Appendix A10 in \cite{Hubbard}.

\end{itemize}

\bigskip

We now mention a couple of other aspects of the infinitesimal theory of \Te space, that will crop up in other parts of this survey.

\subsection*{Variation of extremal length}  
The \textit{extremal length}\index{extremal length} of a homotopy class of a curve $\gamma$ on a Riemann surface $X$  is defined to be
\begin{equation}\label{extl}
\text{Ext}_X(\gamma) = \sup\limits_\rho  \frac{L_\rho(\gamma)^2}{A_\rho}
\end{equation}
where $\rho$ varies over all conformal metrics on $X$ of finite $\rho$-area $A_\rho$, and $L_\rho(\gamma)$ is the infimum of the $\rho$-lengths of all curves homotopic to $\gamma$.  It is also the reciprocal of the \textit{modulus} of $\gamma$, denoted $\text{Mod}(\gamma)$. 

For a path $\{X_t\}_{-1<t<1}$ in $\tg$ with Beltrami differential $\dot{\mu}$ (\textit{c.f.} \eqref{dotmu}), Gardiner proved the following formula\index{Gardiner!formula}:
\begin{equation}\label{gform}
\left.\frac{d}{dt}\right\vert_{t=0} \text{Ext}_{X_t}(\gamma) = 2\text{Re}\langle \phi_\gamma, \dot{\mu}\rangle
\end{equation}
where  $\phi_\gamma$ is a holomorphic quadratic differential, and the pairing on the right is the usual one defined in \eqref{triv}.  
In fact $\phi_\gamma$ is a Jenkins-Strebel differential\index{differential!Jenkins-Strebel} (that we shall define in \S3) that corresponds to $\gamma$ via the theorem of Hubbard-Masur\index{Hubbard-Masur theorem}, that we shall discuss in \S6.
See Theorem 8 of \cite{Gardiner-ext}, and for a more recent, alternative proof, see \cite{Wentworth}.

\subsection*{Metrics}
A Riemannian metric on $\tg$ is defined by specifying a smoothly varying positive-definite inner product on each $T_X^\ast\mathcal{T}_g$, which can be canonically identified with the tangent space at $X$, as in the discussion following \eqref{triv}.   One of the most studied of such metrics is the \textit{Weil-Petersson metric}\index{Teichm\"{u}ller space!Weil-Petersson metric} where the inner product is defined to be 
\begin{equation}\label{wpmetric}
\langle \phi, \psi\rangle_{WP} = \int\limits_X \frac{\phi \bar{\psi}}{\rho} \hspace{.1in}\text{ for each pair } \phi,\psi \in Q(X)
\end{equation}
where $\rho$ is the hyperbolic metric on $X$.  See, for example, \cite{Yamada}, \cite{Wolpert-book} for accounts of various aspects of Weil-Petersson geometry. 

Defining a \textit{norm} instead of an inner-product defines a \textit{Finsler} metric on $\tg$, the most notable of which is the \textit{\Te metric}\index{Teichm\"{u}ller space!Teichm\"{u}ller metric} given by $L^1$-norm
\begin{equation}
\lVert q \rVert = \int\limits_X  \lvert q \rvert \hspace{.1in} \text{ where } q \in Q(X) \text{ }.
\end{equation}
As we shall see  in \S4,  holomorphic quadratic differentials are intimately related to the \textit{global} geometry of \Te space in this metric.

\section{Singular-flat geometry}

A non-zero holomorphic quadratic differential $q$ on a Riemann surface $X$ induces a singular-flat metric on the underlying surface $S$ having the  local expression $\lvert q(z) \rvert \lvert dz \rvert^2$. In fact, $q$ induces a  \textit{half-translation structure}\index{half-translation structure} on $S$,  namely: there is a locally-defined canonical coordinates\index{differential!canonical coordinates} in which $q$ has the expression $d\xi^2$ obtained by the coordinate change \begin{equation}\label{cchange}
z\mapsto \xi := \int\limits^z\sqrt q
\end{equation} 
and two such overlapping charts $\xi_1, \xi_2$ satisfy $\left(\frac{d\xi_1}{d\xi_2}\right)^2 =1$ and hence each translation surface is a \textit{half-translation}, that is, a map of the form $\xi_1 = \pm \xi_2 + c$ where $c\in \C$. 

The singular-flat metric can thus be thought of as the Euclidean metric on the canonical charts, and the  singularities are at the zeroes of the quadratic differential. Indeed, at a zero of the form $z^k dz^2$  the coordinate change $z\mapsto \xi = z^{(k+2)/2}$ is a branched cover of the $\xi$-plane, and the metric acquires a cone point of angle $2\pi(k+1)$.

Conversely, a \textit{half-translation surface} is obtained by starting with planar polygons in $\C$ and identifying sides by half-translations such that the singularities are cone points of the above form. Such a surface acquires a conformal structure $X$ since half-translations are holomorphic, and a holomorphic quadratic differential $q$ since such transition maps preserve the canonical quadratic differential $dz^2$ on each polygon. 

Together, this results in the following:

\begin{thm}\label{flatthm} The space of marked half-translation surfaces of genus $g\geq 2$  is in bijective correspondence with the  total space of the cotangent bundle $T^\ast \mathcal{T}_g$\index{Teichm\"{u}ller space!cotangent bundle} except its zero section. 
\end{thm}

For a parallel discussion of \textit{translation} surfaces corresponding to holomorphic $1$-forms, and its relation to the theory of billiards, see the surveys \cite{Wright, Wright2} and \cite{Zorich}.  For more about the structure, closed trajectories, length-spectra and degeneration of such singular-flat metrics, see for example, \cite{Rafi-thick-thin}, \cite{Eskin-Masur}, \cite{Rafi-Duchin-Leininger}, \cite{Fu} and \cite{Morzadec}.  \\

\subsection*{Jenkins-Strebel differentials}\index{differential!Jenkins-Strebel} 
 A {cylinder} of finite height  is a translation surface obtained by identifying two opposite sides of a rectangle on $\C$;  a special class of half-translation surfaces  in $T^\ast \mathcal{T}_g$ are obtained by taking a collection of such flat cylinders and identifying  boundary intervals by half-translations.  Let  $\Gamma = \{\gamma_1,\gamma_2,\ldots \gamma_k\}$  be the collection of core curves of these cylinders. It turns out the induced quadratic differential metric solves the following extremal problem on the underlying Riemann surface $X$:
 \begin{equation}\label{extprob}
\textit{Minimize }  \sum\limits_{i=1}^kw_i^2 Ext_{A_i}(\gamma_i)  \hspace{.1in} \textit{for a fixed positive tuple }(w_1,w_2,\ldots,w_k),
 \end{equation}
 where $\{A_i\}_{1\leq i\leq k}$ is a partition of $X$ into cylinders with core curves in $\Gamma$,  and $\rho$ is a conformal metric on $X$, that restricts to a conformal metric on each $A_i$. 
  
Here $\text{Ext}_S(\gamma)$ denotes the extremal length\index{extremal length} of $\gamma$ on the subsurface $S$ (see \eqref{extl}), and $\{w_i\}_{1\leq i\leq k}$ are positive real weights that equal the heights of the cylinders in the solution metric.  For details, see Chapter 10 of \cite{Gardiner}, Chapter VI of \cite{Streb} for expository accounts, \cite{Liu} for an alternative proof, and \cite{Jenkins} for the original paper. We shall see a toy version of such an extremal problem when we discuss the Gr\"{o}tzsch argument in \S4.
  
  Such holomorphic quadratic differentials are called \textit{Jenkins-Strebel differentials}, and as a consequence of their simple geometric form, these are an important class of examples in a various contexts (for a recent application, see \cite{Markovic} or \cite{GardCK}).

 \subsection*{Strata}\index{strata}  A holomorphic quadratic differential has $\text{deg}(K^2) = 4g-4$ zeroes (counting with multiplicities), and hence the space of half-translation surface of genus $g\geq 2$ has a stratification, where each strata $Q(\bar{k})$  is defined by  a partition $\bar{k}$ of $4g-4$ which determines  the number and orders of the zeroes (see \cite{Veech-strata} and \cite{Masur-Smillie}).  Note that a \textit{generic}  genus-$g$ half-translation surface has $4g-4$ cone points of angle $3\pi$, and constitute the top-dimensional stratum $\mathcal{Q}({1,1,\ldots1})$ or $\mathcal{Q}(1^{4g-4})$, that is an open dense set in  the total space $T^\ast \mathcal{T}_g$.  The topology of these strata is still not completely known; for work in this direction, primarily concerning the number of components, see  \cite{Lanneau}, \cite{Walker} or \cite{Boissy-end}, and also \cite{Barros}. 
 
 \subsection*{Period coordinates}\index{period coordinates} 
Any quadratic differential $q$  in the top stratum lifts to the square of an \textit{abelian differential} (that is, a holomorphic $1$-form) $\omega$ in the \textit{spectral cover} $\hat{X}$, which is the branched double cover of $X$ branched over the $4g-4$ zeroes  of $q$. The \textit{anti-invariant} or \textit{odd} part of the homology $H_1(\hat{X}, \mathbb{C})^{-}$ has dimension $3g-3$ over $\mathbb{C}$ (\textit{c.f.} \cite{D-H}), and the periods of a basis determine local coordinates for  $\mathcal{Q}(1^{4g-4})$. Moreover, the stratum acquires a volume form that is the pullback of the Lebesgue measure on $\C^{3g-3}$ that is called the \textit{Masur-Veech measure}\index{Masur-Veech measure}, which descends to a finite-volume measure on $T^\ast \mathcal{M}_g$  (see \S5 of \cite{Masur-IE}). The problem of computing these Masur-Veech volumes  is discussed in \cite{Goujard, DGZZ}, and is related to various counting problems and billiards (see, for example, \cite{Athreya-Eskin-Zorich}).

\subsection*{Measured foliations}\index{measured foliation}

A \textit{measured foliation} on a smooth surface $S$ is a one-dimensional foliation that is smooth away from finitely many singularities of ``prong type"\index{measured foliation!prong singularity} (where the defining $1$-form is of the form $\text{Re}(z^kdz)$ for $k\geq 1$)  and is equipped with a transverse measure $\mu$. Here, $\mu$ is a non-negative valued measure on arcs transverse to the foliation, that is invariant under a homotopy that keeps the arc transverse.  Two such measured foliations are considered equivalent if the transverse measures that they induce on the set of simple closed curves on $S$ coincide; topologically, an equivalent pair differs by an isotopy of leaves and ``Whitehead moves" (see Figure 1).

\begin{figure}[h]
	
  \centering
  \includegraphics[scale=0.35]{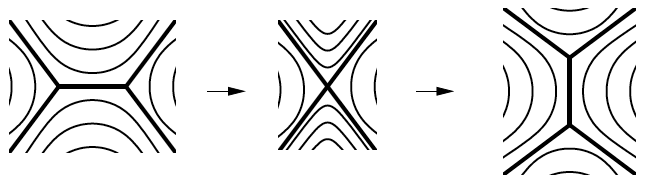}\\
  
  \caption{A Whitehead move preserves the equivalence class of a measured foliation.}
  \label{lune}
\end{figure}

In fact, the transverse measures of  finitely many simple closed curves suffice to uniquely determine an equivalence class of a measured foliation. Indeed, the space $\mathcal{MF}$ of  measured foliations on $S$, up to the equivalence  described above, is homeomorphic to $\mathbb{R}^{6g-6}$. A particularly useful set of coordinate charts are provided by \textit {weighted train-tracks}\index{train track} on the surface. (See  \cite{FLP}, \cite{PenHar} for details.)\\

A holomorphic quadratic differential $q$ on $X$  determines  a pair of such measured foliations on the surface, that we shall now describe. 
Since $q$ determines a  bilinear form  $q: T_xX \otimes T_x X \to \mathbb{C}$ at any point $x\in X$ away from the zeroes, at each such point there is a unique (un-oriented) \textit{horizontal (resp. vertical) direction} $v$ where $q(v,v)\in \mathbb{R}^{+} (\text{resp.} \mathbb{R}^{-} )$ . Integral curves of this line field on $X$ determine the \textit{horizontal  (resp. vertical) foliation} on $X$. 

These  horizontal and vertical foliations are exactly the foliations by horizontal and vertical lines in the canonical coordinates $\xi$ for the quadratic differential (see \eqref{cchange}).  In particular, \textit{at} a zero of order $k$,  the horizontal (and vertical) foliation has a  \textit{$(k+2)$-prong singularity} since it is the pullback of the horizontal  (or vertical) foliation on the $\xi$-plane by the change of coordinates $z\mapsto \xi = z^{k/2 +1}$.

The transverse measure on the  horizontal (resp. vertical) foliation  is defined as follows: in the case $\gamma$ is contained in canonical coordinates $\xi$, the transverse measure of $\gamma$ across the vertical foliation is $\int\limits_\gamma \lvert \text{Re } d\xi \rvert$ (the horizontal distance between the endpoints) and the transverse measure across the horizontal foliation is $\int\limits_\gamma \lvert \text{Im } d\xi \rvert$ (the vertical distance).  In general, one adds such integrals along a covering of $\gamma$ by canonical coordinate charts; this is well-defined as the above integrals are preserved under transition maps (half-translations).

In fact, this pair of transverse measured foliations uniquely recovers the quadratic differential it came from (see \cite{Gard-Mas}). \\

\begin{figure}
	
  \centering
  \includegraphics[scale=0.43]{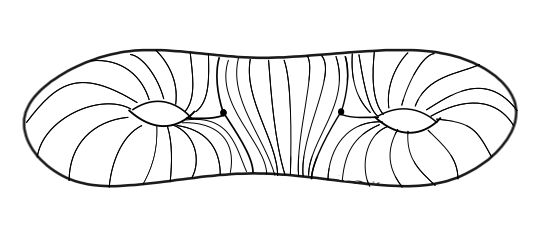}\\
  
  \caption{A Jenkins-Strebel differential induces a measured foliation with all non-critical leaves closed.}
  \label{lune}
\end{figure}

\textit{Example.} For a half-translation surface built from planar polygons as we saw earlier, the horizontal and vertical foliations are respectively induced from the foliations by the horizontal and vertical lines in the plane. In particular, for a Jenkins-Strebel differential, if we assume that each cylinder is obtained by identifying vertical sides of a rectangle, {all non-critical horizontal trajectories are closed}  (a \textit{critical} trajectory is one that starts or ends at a prong-singularity) -- see Figure 2. Such a measured foliation can be represented as a disjoint collection of weighted simple closed curves, that are the core curves of the foliated cylinders, where the weights represent the transverse measures across them.  Indeed, these weighted multicurves are in fact dense in $\mathcal{MF}$; this fact follows easily from the train-track coordinates mentioned earlier.  \\

Trajectory structures for measured foliations\index{measured foliation!trajectory structure} induced by quadratic differentials was a significant area of research (see  \cite{Jenkins}) until  the work of Hubbard and Masur\index{Hubbard-Masur theorem}, that we shall discuss in \S6,  showed that the space of such structures coincides with the space  $\mathcal{MF}$, that was defined by Thurston in a different context (see \cite{FLP}). Thurston showed that the space of \textit{projectivized} measured foliations  $\mathbb{P}\mathcal{MF}$\index{measured foliation!projectivized}\index{Thurston!compactification} (obtained by scaling the measures by a positive real factor) determines a compactification $\overline{\mathcal{T}_g} = \mathcal{T}_g \sqcup \mathbb{P}\mathcal{MF}$,  such that the action of the mapping class group extends continuously to this boundary.  For the equivalent notion of \textit{measured laminations}\index{measured lamination} and the corresponding space $\mathcal{ML}$,  and a proof of the equivalence, see \cite{Kapovich}.

\section{Extremal maps and \Te geodesics} 

A quasiconformal (or qc-) map $f:X\to Y$\index{quasiconformal!map} between a pair of Riemann surfaces is a homeomorphism with a weak (distributional) derivative that is locally square-integrable, such that the Beltrami differential  
\begin{equation}\label{Beltrami}\index{Beltrami!equation}
\mu =  f_{\bar{z}}/f_z
\end{equation}
that is defined almost everywhere, satisfies $\lVert \mu \rVert_\infty  < 1$. Alternatively, the map $f$ has bounded \textit{dilatation} $K(f) := (1 + \lVert \mu \rVert_\infty)/(1 - \lVert \mu \rVert_\infty )$\index{quasiconformal!dilatation}. \\

\textit{Example.} An affine stretch map $f(x,y) = (x,\kappa y)$ is a quasiconformal map between two planar rectangles: the unit square $R_1$ and a rectangle $R_2$ of width one and height $\kappa $. A short computation (after converting to $z$ and $\bar{z}$ coordinates) shows that $K(f) = \kappa$.\\

When a Beltrami differential $\mu$ satisfying $\lVert \mu \rVert_\infty  < 1$ is specified, \eqref{Beltrami}  is called the \textit{Beltrami equation}\index{Beltrami!equation} for $f$.  The Measurable Riemann  Mapping Theorem\index{measurable Riemann mapping theorem} asserts that in the case $X=Y = \hat{\C}$, there always exists a solution to the Beltrami equation that is unique up to post-composition with a M\"{o}bius transformation. 

For a precise statement, see Chapter V of  \cite{Ahlfors}, and the other chapters there for a thorough treatment of the theory of quasiconformal maps. We highlight the following elementary, but key, lemma of Gr\"{o}tzsch (see Chapter I.B of \cite{Ahlfors}):

\begin{lem}[{Gr\"{o}tzsch}]\label{grz}\index{Gr\"{o}tzsch argument} The affine stretch map $f:R_1 \to R_2$ as above, has the \textit{least} quasiconformal dilatation amongst all quasiconformal maps from $R_1$ to $R_2$ that map vertices to vertices.
\end{lem}
\begin{proof}
The proof uses what is known as the \textit{length-area} method. Since post-composing $f$ by a conformal dilation does not change its Beltrami differential, we can assume that the rectangles have the same height, namely, $R_i = \{(x,y) \text{ } \vert \text{ }0\leq x\leq L_i \text{ },\text{ } 0 \leq y\leq H\}$ for $i=1,2$. 
Moreover, we can also assume that $L_2\geq L_1$, by interchanging the horizontal and vertical sides if necessary.

Let $f$ be the least-dilatation quasiconformal map between them. Since $f$ is a homeomorphism, the image of any horizontal line  between the vertical sides of $R_1$ is some arc between the vertical sides of $R_2$.
In particular, we have:
\begin{equation*}
L_2 \leq \displaystyle\int_0^{L_1} f_x(t, y) dx  \leq \displaystyle\int_0^{L_1} (\lvert f_{\bar{z}}\rvert + \lvert f_z\rvert ) dx 
\end{equation*}
for any $y\in [0,H]$.
Integrating in the $y$-direction, we then obtain:
\begin{equation*}
L_2H \leq \displaystyle\int_0^{H}\displaystyle\int_0^{L_1} (\lvert f_{\bar{z}}\rvert + \lvert f_z\rvert )  dxdy 
\end{equation*}
Applying the Cauchy-Schwarz inequality we have:
\begin{equation*}
L_2^2H^2 \leq  \displaystyle\int_{R_1} \left(\frac{1+ \lvert\mu\rvert}{1-\lvert\mu\rvert}\right) dxdy  \displaystyle\int_{R_1} (\lvert f_{\bar{z}} \rvert^2 - \lvert f_z\rvert^2) dxdy
\end{equation*}
where $\mu$ is the Beltrami differential as in \eqref{Beltrami}.
Since the integrand in the second term on the right hand side is the Jacobian of $f$, the change of variables formula yields that the second factor on the right equals $Area(R_2) = L_2H$. After rearranging terms, we obtain
\begin{equation*}
\frac{L_2}{L_1} \leq \frac{1}{Area(R_1)} \displaystyle\int_{R_1} \left(\frac{1+ \lvert\mu\rvert}{1-\lvert\mu\rvert}\right) dxdy.  
\end{equation*} 
The right hand side is the \textit{average} value of the integrand over $R_1$, and we obtain in particular, that
\begin{equation*}
\frac{L_2}{L_1} \leq \sup\limits_{R_1} \left(\frac{1+ \lvert\mu\rvert}{1-\lvert\mu\rvert}\right). 
\end{equation*} 
It can easily be checked that in fact, equality holds for the map 
\begin{equation*}
	f(z) = \frac{L_2}{L_1}\left(\frac{z + \bar{z}}{2}\right) +  \left(\frac{z-\bar{z}}{2}\right)
\end{equation*}
that is an affine stretch in the $x$-direction.
\end{proof}

\medskip

Historically, the first \textit{global} parametrization of Teichm\"{u}ller space\index{Teichm\"{u}ller space!Teichm\"{u}ller's parametrization} with holomorphic quadratic differentials arose in the context of  such {extremal maps}\index{Teichm\"{u}ller!map}:

\begin{thm}[Teichm\"{u}ller]\label{param}\index{Teichm\"{u}ller!theorem} 
For any pair of marked Riemann surfaces $X,Y\in \tg$ there is a unique quasiconformal map (the \text{\Te map}) $F:X\to Y$ that minimizes $K(f)$ over all quasiconformal maps $f:X\to Y$ that preserve the marking. 

The map $F$ is obtained by an affine stretch map 
\begin{equation}\label{stretch}
(x,y) \mapsto (\kappa^{-1/2}x,\kappa^{1/2}y) \hspace{.1in} \text{where } \hspace{.1in} \kappa > 1
\end{equation} 
on the canonical coordinates $\xi = x+iy$  of some holomorphic quadratic differential $q$ on $X$.

Moreover, for any such constant $\kappa>1$  the map 
\begin{equation*}
\Psi:  Q(X) \to \tg 
\end{equation*} 
is a homeomorphism, where $\Psi(q) = X_{q, \kappa}$, the marked Riemann surface obtained by an affine stretch of the canonical coordinates of $q$ as above. 
\end{thm}

The proof of the first part is a more elaborate version of the argument\index{Gr\"{o}tzsch argument}  we saw in the proof of Lemma \ref{grz}  -- see for example, Chapter 11 of \cite{FarbMarg}, or Chapter 4.6 of \cite{Gard-Lakic}. Indeed, the proof given in \cite{Gard-Lakic} relies on the \textit{Reich-Strebel inequality} (\cite{ReichStreb}, and for more about its relation to the {Gr\"{o}tzsch} argument, see \cite{MarkMat}.   To show that $\Psi$ is a homeomorphism, it suffices to prove that the map is proper, injective and continuous, by Brouwer's Invariance of Domain.  Properness is a consequence of the fact that a quasiconformal map can only distort hyperbolic lengths by a bounded multiplicative factor (Lemma 3.1 of \cite{Wolpert-lem}), and injectivity is the uniqueness part of Teichm\"{u}ller's theorem.  The proof of continuity of the map is not obvious,  however:  but as one varies $q$ continuously, so do the  Beltrami differentials, and one then uses the continuous (in fact, analytical) dependence on parameters for solutions of the  Measurable Riemann Mapping Theorem.

\subsection*{\Te rays} 
The \textit{\Te map} $F$\index{Teichm\"{u}ller!map} in Theorem \ref{param} achieves the {\Te distance} between the pair $X$ and $Y$ defiined by
\begin{equation*}
d_{\mathcal{T}}(X,Y) := \frac{1}{2} \inf\limits_{f:X\to Y}\log K(f) 
\end{equation*} 
where $f$ are quasiconformal homeomorphisms that preserve the markings. 
Kerckhoff gave an alternative formula for this distance in terms of extremal length:
\begin{equation*}
d_{\mathcal{T}}(X,Y) = \frac{1}{2} \inf\limits_{\gamma\in \mathcal{S}}\log  \frac{\text{Ext}_X(\gamma)}{\text{Ext}_Y(\gamma)} 
\end{equation*} 
where $\mathcal{S}$ is the set of all simple closed curves on the underlying topological surface (see \cite{Kerck}).

It is then a consequence of Gardiner's formula for the variation of extremal length\index{extremal length} \eqref{gform}, that this distance is exactly the one that arises from the \Te metric defined in \S2.

Indeed, the one-parameter family of surfaces  $X_{q, t}$ (for $t\geq 0$)  obtained by \Te maps with stretch-factors $\kappa(t) = e^{t}$ is a path in the cotangent bundle $T^\ast \tg$  via the identification in Theorem \ref{flatthm}, and  is a \textit{geodesic ray}\index{Teichm\"{u}ller!ray} from the initial point $(X,q)$.  For details, see Chapter 9.3 of \cite{Lehto} or  Chapter 7 of \cite{Gardiner}.

\subsection*{\Te dynamics} 
The preceding geometric description of geodesics in the \Te metric is extremely useful, and is notably absent for, say, Weil-Petersson geodesics, which makes it the latter more difficult to study.   In particular, one can show that $\tg$ equipped with the \Te metric is a complete and geodesic metric space. Jenkins-Strebel\index{differential!Jenkins-Strebel} quadratic differentials determine ``parallel" geodesic rays that remain a bounded distance from each other, and hence the \Te metric is not negatively curved (\cite{Masur}, \cite{MasurWolf}), though it does have various features of negative curvature (see, for example, \cite{DuchinMasur}). For a survey concerning the geometry of $\tg$ with the \Te metric, see \cite{Masur-survey}. 

Indeed, the seminal work of Masur and, independently, Veech  showed that that \Te geodesic flow on $T^\ast \mathcal{M}_g$ is ergodic (\cite{Masur-IE}, \cite{Veech}) with respect to the Masur-Veech measure mentioned in \S3. This has had an enormous influence in \Te theory, with myriads of applications, especially to various counting problems, starting with \cite{Veech-Siegel} and \cite{Eskin-Masur}.\index{Teichm\"{u}ller space!geodesic flow}

The affine stretch maps that determine a \Te geodesic ray can be thought of as being obtained by the action of the diagonal subgroup of $\slr$; the orbit of the entire group $\slr$ is a \textit{\Te disk}\index{Teichm\"{u}ller!disk}. The $\slr$-action on $T^\ast \mg$ has been  studied  recently at great depth, developing the analogy with similar actions on homogeneous spaces, and the program of classifying orbit closures, that started with  the work of McMullen (\cite{McMullen-genus2}), includes the work of Eskin-Mirzakhani-Mohammadi (\cite{E-M-M}, see also \cite{Wright2} for an expository account).

\section{Hopf differentials of harmonic maps}

A \textit{harmonic map }$h:X \to Y$\index{harmonic map} between two Riemann surfaces equipped with conformal metrics $(X, \sigma\lvert dz \rvert^2)$ and $(Y,\rho\lvert dw\rvert^2)$ is a critical point of the energy functional 
\begin{equation*}
\mathcal{E}(f) =  \frac{1}{2}\displaystyle\int\limits_X  \lVert df\rVert^2 \sigma dzd\bar{z}
\end{equation*}
where the integrand is the \textit{energy density}, also written as $e(f)$. 
Equivalently,  $h$ satisfies the following \textit{harmonic map equation} :
\begin{equation}\label{harm}
\Delta h +  (\log \rho)_w h_z h_{\bar{z}} = 0
\end{equation}
in any local chart, which is obtained as an Euler-Lagrange reformulation. 

When the target surface $Y$ is strictly  negatively curved (\textit{e.g.} a hyperbolic metric), then such a map exists in the homotopy class of any diffeomorphism (\cite{E-S}). Moreover, such a harmonic map is unique and consequently energy-minimizing (\cite{Hartman}), and is itself also a diffeomorphism (\cite{Sampson}, \cite{Schoen-Yau}). 
The pullback of the metric on $Y$ has the local expression:
\begin{equation}\label{pullb} 
h^\ast \rho = qdz^2 + \sigma e(h) dzd\bar{z} + \bar{q} d\bar{z}^2
\end{equation}
and the \textit{Hopf differential}\index{differential!Hopf} is defined to be $ qdz^2$, which is the $(2,0)$-part of the above metric. \\

The relation with  holomorphic quadratic differentials arises from the fact:

\begin{lem} A diffeomorphism $h:X\to Y$ is harmonic if and only if its Hopf differential $q$ is a holomorphic quadratic differential on $X$.
\end{lem}
\begin{proof}
By virtue of \eqref{pullb}, we have $q = \rho h_z \bar{h}_z dz^2$.
One can check that by an application of the Chain Rule that 
\begin{equation*}
\frac{\partial}{\partial \bar{z}} ( \rho h_z \bar{h}_z ) = h_z \overline{\text{EL}} + \bar{h}_z{\text{EL}} 
\end{equation*}
where $\text{EL}$ is the left hand side of \eqref{harm}. 

If $h$ is harmonic, $\text{EL} =0$ and hence $\frac{\partial}{\partial \bar{z}} ( \rho h_z \bar{h}_z ) = 0$
that shows that $q$ is a holomorphic quadratic differential.

Conversely,  if $q$ is holomorphic, we have 
\begin{equation*}
 h_z \overline{\text{EL}} + \bar{h}_z{\text{EL}} =0
\end{equation*}
and if $\text{EL} \neq 0$ at a point $z_0$, this would imply $\lvert h_z\rvert = \lvert \bar{h}_z \rvert = \lvert h_{\bar{z}} \rvert$ and hence the Jacobian 
\begin{equation*}
J(z_0) = \lvert h_{\bar{z}} \lvert^2 - \lvert h_{{z}} \lvert^2 =0
\end{equation*}
which contradicts the fact that $J \neq 0$ everywhere (since $h$ is a diffeomorphism).
\end{proof}

\subsection*{Wolf's parametrization of $\tg$} 
Like for extremal quasiconformal maps in the previous section, harmonic maps provide a parametrization of \Te space:

\begin{thm}[Wolf, \cite{Wolf0}]\label{wparam}\index{Teichm\"{u}ller space!Wolf's parametrization}\index{Wolf's parametrization} Fix a basepoint $X \in \tg$. Let $$\Phi: \tg \to Q(X)$$ be map that assigns to any marked hyperbolic surface $Y$, the Hopf differential of the unique harmonic map $h_Y: X\to Y$ that preserves the markings.  Then $\Phi$ is a homeomorphism.
\end{thm}

The proof uses the following B\"{o}chner-type equation that can be derived from \eqref{harm}:
\begin{equation}\label{boch}
\Delta u = e^u - \lvert q\rvert^2 e^{-u}
\end{equation}
where $u = \log \lVert \partial_z h\rVert$. This equation can in fact be shown to be \textit{equivalent} to the harmonic map equation \eqref{harm};  namely, lifting to the universal cover, it is the Gauss-Codazzi equation for a constant-curvature spacelike immersion in Minkowski space $\mathbb{R}^{2,1}$, and the Gauss map of this immersion recovers the (equivariant)  harmonic map to the hyperbolic plane (see \cite{Wan}). 

As in Theorem \ref{param}, the strategy of the proof of Wolf's parametrization  is to show that the map $\Phi$ is proper, injective and continuous.
The injectivity is a consequence of an argument using the maximum principle on \eqref{boch} and this time, continuity is immediate. The proof of properness relies on expressing the energy density in terms of the Hopf differential, and showing that the \textit{energy} of the harmonic map $h_Y$ is a proper function as $Y$ varies in \Te space (see \cite{Goldman-Wentworth} or \cite{Toledo}, for generalizations of this).

\subsection*{Hopf differential and the image of the harmonic map}\index{harmonic map!estimates} 
In the canonical coordinates $\xi = x+iy$ for the Hopf differential $q$,  the metric pullback \eqref{pullb} becomes:
\begin{equation}
h^\ast (\rho) = (e(h)+2) dx^2 + (e(h)-2)dy^2
\end{equation}
which implies that the horizontal (resp. vertical) directions are the directions of maximum (resp. minimum) stretch for the map $w$.

A finer estimate by an  analytical argument  examines \eqref{boch} and proves an exponential decay of the energy density $e(h)$ away from the zeros of $q$, thus relating the behaviour of the harmonic map with the singular-flat geometry of the Hopf differential. This yields the following statement --  see \S3.3 of \cite{Minsky},  \cite{Wolf-High} for details, \cite{Han-Remarks} for an exposition, and \cite{Dumas-Wolf} for a sharper estimate. 

\begin{prop}\label{gest}
Let $h:X\to Y$ be a harmonic diffeomorphism where $Y$ is a surface equipped with a hyperbolic metric $\rho$. Let $\gamma_{horiz}$ be a horizontal trajectory on $X$ with respect to the singular-flat structure induced by  $q$, and let $\gamma_{vert}$ be a vertical trajectory.

Then there exists constants $C, \alpha>0$ such that 

\begin{itemize}
\item if $\gamma_{horiz}$ is of length $L$, then it  is mapped to an arc on $Y$ of geodesic curvature less than $Ce^{-\alpha R}$ and of hyperbolic length $ 2L \pm Ce^{-\alpha  R}$, and

\item if  $\gamma_{vert} $  is of length $L$, then it is mapped to an arc of hyperbolic  length  less than $L \cdot Ce^{-\alpha  R}$,
\end{itemize}
where $R>0$ is the radius of an embedded disk (in the $q$-metric) containing either segment, that does not contain any singularity.
\end{prop}

\medskip 

We shall see some examples in \S8.3, in the context of harmonic maps from $\mathbb{C}$ to the Poincar\'{e} disk. \\

As a consequence of this estimate, Wolf showed that his parametrization extends to the Thurston compactification\index{Thurston!compactification} by projectivized measured foliations, that we mentioned in \S3.

\subsection*{Non-abelian Hodge correspondence}\index{non-abelian Hodge correspondence}  
Wolf's parametrization of \Te space is equivalent to that in Hitchin's seminal work (\cite{Hitchin} -- see \S11 of that paper) which was subsequently subsumed in a much more general correspondence, proved by Simpson (see \cite{Simpson}). In particular, there is a Kobayashi-Hitchin correspondence between the variety of surface-group representations in $\mathrm{SL}_n(\mathbb{R})$ up to conjugation, and the ``Hitchin section" of the moduli space of stable rank-$n$ Higgs bundles on a Riemann surface $X$.  A key intermediate step is the existence of equivariant harmonic maps from the universal cover of $X$, to the symmetric space  $\mathrm{SL}_n(\mathbb{R})/\mathrm{SO}(n)$, which for $n=2$ is precisely the hyperbolic plane, that was discussed in this section. 
For a survey focusing on this interaction of harmonic maps and \Te theory,  see \cite{Went-Daska}.

\section{The Hubbard-Masur Theorem}

It is a consequence of classical Hodge theory that the space of holomorphic differentials  ($1$-forms) on a Riemann surface $X$ can be identified with the first cohomology of the underlying smooth surface with real coefficients:
\begin{equation*}
H^0(X, K) \cong H_1(S, \mathbb{R})^\ast = H^1(S, \mathbb{R})
\end{equation*}
where the identification is via the real parts of periods, namely,  $\omega \mapsto  \displaystyle\int\limits_{\gamma_i}  \text{Re}   (\omega)$ where $\gamma_i$ varies over a basis of homology. The analogue of this identification for holomorphic \textit{quadratic} differentials, is the following theorem of Hubbard and Masur (\cite{HM}), where the objects on the topological side are measured foliations,  that we introduced in \S3.

\begin{thm}[Hubbard-Masur]\label{hm}\index{Hubbard-Masur theorem} Let $X\in \tg$ where $g\geq 2$.
The map $\Phi_{HM}:Q(X) \to \mathcal{MF}$ that assigns to a holomorphic quadratic differential  its vertical measured foliation, is a homeomorphism.
\end{thm}

\textit{Remark.} The same holds, of course, for the horizontal measured foliation, but the sketch of the proof at the end of the section is most natural for the vertical one. \\

As discussed in \S3, two measured foliations are equivalent if the induced transverse measures are the same. Moreover,  for the vertical foliation, the  measure of a transverse loop $\gamma$ is $\displaystyle\int\limits_{\gamma}  \lvert  \text{Re} (\sqrt q) \rvert$, and thus these transverse measures can thus be thought of as encoding ``periods" of the quadratic differential. \\

The original proof involved a construction of a continuous {section} $\sigma_F$ of the cotangent bundle $T^\ast \tg$, for each measured foliation $F\in \mathcal{MF}$, such that $\sigma(X) \in Q(X)$ has an induced vertical foliation $F$. The case when the foliation $F$ had prong-type singularities of higher order involved a delicate analysis of ``perturbations" into a generic singularity set.  A shorter proof was given by Kerckhoff  (\cite{Kerck}, see also \cite{Gardiner} for an exposition) where he used the existence (and uniqueness) of Jenkins-Strebel differentials having one cylinder of prescribed height.  Both these proofs used the fact that such special Jenkins-Strebel differentials are dense in $T^\ast \tg$, as proved by Douady and Hubbard (\cite{D-H}).\\

Wolf's proof uses harmonic maps to $\mathbb{R}$-trees (see \cite{Wolf2, WolfT}); it is this strategy, that we briefly sketch below, that was employed in the generalization of the Hubbard-Masur theorem to meromorphic quadratic differentials, that we shall discuss in \S8. An equivalent approach was developed by Gardiner and others (see \cite{Gard1,Gard-Lakic}) in which the holomorphic quadratic differential on $X$ emerges as the solution to an extremal problem, similar to the case of Jenkins-Strebel differentials. \\

We begin with two definitions:

First, an \textit{$\mathbb{R}$-tree}\index{$\mathbb{R}$-tree} is a geodesic metric space, such that any two points has a unique embedded path between them that is isometric to a real interval; this generalizes the notion of a simplicial tree (and in particular, may have locally infinite branching). In our context, the leaf-space $T_F$ of the lift of a measured foliation $F$ to the universal cover is an $\mathbb{R}$-tree (see Chapter 11 of  \cite{Kapovich} or \cite{Otal} for an exposition).

Second, an equivariant map from the universal cover of a Riemann surface  to an $\mathbb{R}$-tree is \textit{harmonic} if convex functions pull back to subharmonic functions (see \cite{Went-Daska}). In particular, there is a well-defined Hopf differential, which is holomorphic, as in the case of smooth targets. This forms a special case of a deeper theory of harmonic maps to such non-positively curved metric spaces, developed by Korevaar-Schoen (\cite{KorSch1, KorSch2}).

\subsection*{Strategy of Wolf's proof}

Fix a compact Riemann surface $X$ on which we aim to realize a measured foliation $F$. Passing to the universal cover, one can show that there is a $\pi_1(X)$-equivariant harmonic map from $\tilde{X}$ to $T_F$, the $\mathbb{R}$-tree dual to the lift of $F$.
The proof of this existence theorem proceeds by showing that an energy-minimizing sequence is equicontinuous; the uniform boundedness, in particular, is a consequence of a finite energy bound that we have because of compactness of $X$.

The preimage of a point on $T_F$ is a vertical leaf of the Hopf differential $\tilde{q}$, since it is clearly along directions of maximal collapse. Recall that each point of $T_F$ also corresponds to a leaf of $\tilde{F}$; a topological argument shows that in fact  the vertical foliation of $\tilde{q}$ is precisely the lift of $F$. By the equivariance of the map, $\tilde{q}$ descends to the desired holomorphic quadratic differential on $X$.

\section{Schwarzian derivative and projective structures}

A \textit{(complex) projective structure}\index{projective structure} on a surface $S$ is an atlas of charts to $\cp$ such that the transition maps on the overlaps are restrictions of M\"{o}bius transformations. Piecing together the charts by analytic continuation in the universal cover $\tilde{X}$, one can define a developing map $f:\tilde{X} \to \cp$ that is equivariant with respect to the holonomy representation $\rho:\pi_1(S) \to \pslc$. The space $\mathcal{P}_g$ of marked projective structures on a genus $g\geq 2$ surface $S$  (up to isotopy) admits a forgetful projection map $$p:\mathcal{P}_g \to \tg$$ that records the underlying Riemann surface structure. 

Such structures arise in connection with the uniformization theorem -- indeed, hyperbolic surfaces provide examples of projective structures which have \textit{Fuchsian} holonomy, and a developing image that is a round disk $\D \subset \cp$.  Other examples of projective structures include a \textit{quasi-Fuchsian} structure\index{projective structure!quasi-Fuchsian} in which the developing image is a {quasidisk} $\Omega$,  and the holonomy is a discrete subgroup of $\pslc$.  See \cite{Kra-Maskit} or \cite{Shiga} for more on the relation of projective structures to Kleinian groups. \\

The fibers of the projection map $p$ can be parametrized by holomorphic quadratic differentials; indeed, we have:

\begin{prop}\label{prop9}  The set of projective structures on a fixed marked Riemann surface $X \in \tg$ is parametrized by $Q(X)$.
\end{prop}

As we shall see below, this identification with $Q(X)$ is not a canonical one, but depends on a choice of a base projective structure $X_0$ (for example, the uniformizing Fuchsian structure)  with holonomy  $\Gamma_0$, that shall correspond to the zero quadratic differential.

\begin{proof}[Sketch of the proof of Proposition \ref{prop9}]

In one direction,  consider the \textit{Schwarzian derivative}\index{Schwarzian derivative}  of the developing map $f$, defined as:
\begin{equation*}
S(f) =  \left(\frac{f^{\prime\prime}}{f^\prime} \right)^\prime  - \frac{1}{2} \left(\frac{f^{\prime\prime}}{f^\prime} \right)^2  \text{ }.
\end{equation*}
This  measures the deviation of any univalent holomorphic map $f$ from being a M\"{o}bius map.
  In particular, one can verify that whenever  $A$ is a M\"{o}bius map, that is, an element of $\pslc$, we have:
\begin{center}
$S(A) \equiv 0$,  $S(A\circ f) = S(f)$, and $S(f \circ A)(z)  = S(f) \circ A(z) A^\prime(z)^2 \text{ }.$
\end{center} 
This implies, in our setting, by the $\rho$-equivariance of the developing map $f$, that the holomorphic quadratic differential $S(f)dz^2$ on $\tilde{X}$ is invariant under the action of $\Gamma_0$ and descends to the Riemann surface $X$.

In the other direction, given a quadratic differential $q \in Q(X)$,  we can define a corresponding projective structure by considering the \textit{Schwarzian equation} on $\tilde{X}$ given  by 
\begin{equation}\label{schw2} 
u^{\prime\prime} + \frac{1}{2}\tilde{q}u = 0
\end{equation}
where $\tilde{q}$ is the lift to the universal cover.  The ratio $f:= u_1/u_2$  of two linearly independent solutions $u_1$ and $u_2$  then defines the developing map to $\cp$, and the monodromy along loops on $X$ yields the holonomy representation $\rho$.

  By \eqref{schw2} the Wronskian $u_1u_2^\prime - u_1^\prime u_2$ is constant, and it is an exercise to show that the Schwarzian derivative $S(f)$ recovers $q$, hence establishing that the constructions in the two directions are in fact inverses of each other.
\end{proof}

\begin{cor} The space of projective structures  $\mathcal{P}_g$ is an affine bundle over $\tg$ modelled on the vector bundle $Q_g \to \tg$\index{Teichm\"{u}ller space!co-tangent bundle}.
\end{cor}

Projective structures in fact gives rise to a \textit{global} realization of $\mathcal{T}_g$ as a bounded complex domain in $Q(X)$.  This relies on the following deep observation of Bers (see \cite{Bers-Simul}) concerning the quasi-Fuchsian projective structures that we defined above:

\begin{thm}[Simultaneous Uniformization]\index{Bers!simultaneous uniformization} For any pair $(X,Y) \in \tg\times \tg$, there is a unique quasi-Fuchsian structure $QF(X,Y)$ determined by a discrete subgroup $\Gamma < \pslc$ that leaves invariant a quasi-circle $\Lambda_\Gamma$ in $\hat{\C}$ with complementary components $\Omega_\pm$, such that $\Omega_+/\Gamma = X$ and $\Omega_-/\Gamma = Y$. 
\end{thm}

Fixing an $X\in \tg$, and varying $Y$ over $\tg$, then defines the \textit{Bers embedding}\index{Bers!embedding} 
\begin{equation*}
B_X:\tg \to Q(X) \cong \C^{3g-3}
\end{equation*}
by taking the holomorphic quadratic differential corresponding to the projective structure $QF(X,Y)$. 

For more details, and various open questions about this embedding, see the survey \cite{GupSesh-survey}. For a more comprehensive treatment of this approach to \Te theory, see the books \cite{Nag, Ima-Tani}. \\

\subsubsection*{Grafting} To conclude this section,  we mention that Thurston showed that  any projective structure is obtained from a Fuchsian one by a geometric operation of \textit{grafting}\index{grafting} along a measured lamination $\lambda$\index{measured lamination} on a hyperbolic surface $X$. We describe this operation only in the case that $\lambda$ is a simple closed geodesic with weight $t$: here, a projective annulus $A_t$ is grafted into $X$ along the curve, where $A_t$ is the quotient of the wedge $\{ 0 \leq  Arg(z) \leq t\} \subset \C$ by the hyperbolic element that is the holonomy of the curve. More generally, a measured lamination can be approximated by weighted multicurves, and a grafted surface can be defined in the limit. The resulting surface $gr(X, \lambda)$ is a new complex projective surface; in fact, Thurston showed\index{Thurston!grafting theorem} that the map $gr: \tg \times \mathcal{ML} \to \mathcal{P}_g$ is a homeomorphism. See \cite{KamTan} and \cite{Tanig} for an idea of the proof, and \cite{KulPink} for a more general context.

  By scaling the measure on the lamination $\lambda$ by $t\geq 1$, one obtains a one-parameter family of projective structures that project to \textit{grafting rays}\index{grafting!ray} in $\mathcal{T}_g$, that are in fact asymptotic to \Te geodesic rays\index{Teichm\"{u}ller!ray} that we saw in \S4 (see \cite{Gup1}). 

Furthermore, the work of Dumas (\cite{Dum2}) compares the Schwarzian derivatives of the corresponding projective structures to the holomorphic quadratic differentials corresponding to $t\lambda$ via the Hubbard-Masur theorem.   For more on this, and the topics mentioned in this section, see  the extensive survey on projective structures by Dumas \cite{Dum0}.

\section{Meromorphic quadratic differentials}

The preceding sections were under the assumption that the underlying Riemann surface was compact; in this section we shall consider the case of a Riemann surface $X$ of finite type, that is, of finite genus and finitely many punctures, such that the Euler characteristic is negative.  

Throughout, we shall assume that the  quadratic differentials on $X$ are meromorphic with \textit{finite} order poles at the punctures. Here, a pole \textit{of order $n\geq 1$} means that the expression of the quadratic differential $q$  is 
\begin{equation}\label{localexp}
q = \left(a_nz^{-n} + a_{n-1}z^{1-n} + \cdots + a_1z^{-1} + a_0 + g(z) \right) dz^2
\end{equation}
in a choice of a coordinate disk $\mathbb{D}$ around the pole, where $a_i \in \mathbb{C}$  for $0 \leq i\leq n$ and $g(z)$ is a holomorphic function that vanishes at the origin (where the pole is). 

\subsection{Simple poles} 

A pole of order one is also called a ``simple" pole; in this case the induced singular-flat metric is still of finite area, since $\text{Area}_q(X) =\displaystyle\int_X \lvert q \rvert = \lVert q\lVert_{L^1(X)}$ and $\lVert z^{-1}{dz}^2\rVert_{L^1(\mathbb{D})}  < \infty$.
Indeed, the double cover $w\mapsto z=w^2$ branched at the simple pole  pulls back $q$ as in \eqref{localexp} to the  holomorphic quadratic differential  $4\left(a_1 + a_0w^2 + w^2g(w^2)\right)dw^2$, and much of the analysis of this case reduces to that for holomorphic quadratic differentials (on a compact surface)  via this branched cover. 

The space of such meromorphic quadratic differentials arises as the tangent space of the \Te space $\mathcal{T}_{g,n}$ of a surface of finite type, having genus $g$ and $n$ punctures, and the discussion in the previous sections can be carried forth in this case. This was in fact, the original setting of the work of \Te (see the commentary in \cite{Teich-comm}). 

In particular, a measured foliation\index{measured foliation!at a simple pole} induced by such a meromorphic quadratic differential has a ``fold" at each simple pole -- see Figure 3(a); the corresponding space of such measured foliations can be parametrized by train-tracks ``with stops"\index{train track!with stops} (see Chapter 1.8 of \cite{PenHar}) and the analogue of the Hubbard-Masur theorem  (Theorem \ref{hm}) holds. Considering $\mathcal{T}_{g,n}$ to be the space of marked hyperbolic surfaces of genus $g$ and $n$ cusps, the analogue of Wolf's parametrization (Theorem \ref{wparam}) is an early work of Lohkamp (see \cite{Lohk}). \\

\begin{figure}
	
  \centering
  \includegraphics[scale=0.35]{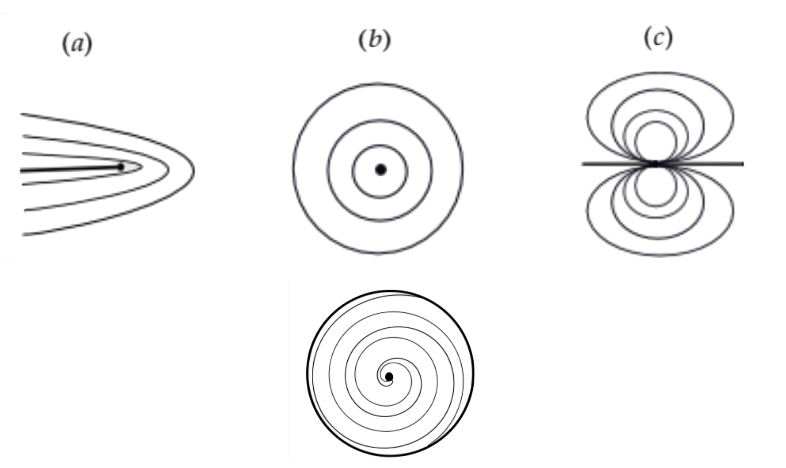}\\
  
  \caption{Structure of the induced horizontal foliation near (a) a simple pole, (b) a pole of order two, and (c) pole of order $4$.}
  \label{lune}
\end{figure}

\subsection{Poles of order two} 

In this case the induced singular-flat metric around the puncture has the form $\frac{dr^2 + rd\theta^2}{r^2}$ in polar coordinates, and is hence a semi-infinite Euclidean cylinder, and in particular, of infinite area. One way such a pole arises is in a limit of a sequence of  Jenkins-Strebel differentials\index{differential!Jenkins-Strebel} where the length of one of the Euclidean cylinders in the induced metric diverges to infinity.   This happens, for example, along \textit{Strebel rays}, that are \Te rays determined by a Jenkins-Strebel differential. 

The corresponding limiting quadratic differential is called a \textit{Strebel differential}\index{differential!Strebel}\index{Strebel!differential} (though sometimes it is also called a Jenkins-Strebel differential in the literature), which is meromorphic quadratic differential  with poles of order two, with the additional property that \textit{each non-critical horizontal leaf is closed}.

Indeed, we have the following existence (and uniqueness) theorem:

\begin{thm}[Strebel]\index{Strebel!theorem} Let $X$ be a compact Riemann surface of genus $g\geq 1$ and let $P = \{p_1,p_2,\ldots p_k\}$ be a non-empty set of points. Given positive real numbers $a_1,a_2,\ldots a_k$, there is a unique Strebel differential with poles of order two at the points of $P$, and such that for each $1\leq i\leq k$, the semi-infinite Euclidean cylinder at $p_i$  has circumference $a_i$.
\end{thm}

The proof involves showing that a Strebel differential arises as the solution of an extremal problem, similar to that of Jenkins-Strebel differentials  -- see \eqref{extprob} --  but involving, instead of the extremal lengths of the curves, the  (reciprocal of the) \textit{reduced modulus} of the punctured disks centered at the points of $P$.  
See \S4.1.2 of \cite{Mondello} for a sketch of the proof, and Chapter VI \S23 of \cite{Streb} or \cite{Arba-Corn} for a longer exposition.

The {critical} trajectories of a Strebel differential form a metric graph $G$  that is a retract of the punctured surface, as can be seen  by collapsing each semi-infinite cylinder in the vertical direction. The orientation of the surface also induces a \textit{ribbon graph} structure on the embedded graph $G$, namely, an orientation on the half-edges emanating from each vertex. 
By Strebel's theorem, there is a unique such metric ribbon graph associated  to any Riemann surface with punctures $(X,P)$ with a set of positive real parameters at each puncture; this gives a combinatorial description of decorated \Te space\index{Teichm\"{u}ller space!decorated} $\mathcal{T}_{g,n}\times (\mathbb{R}^+)^n$ that has been useful in various contexts (see for example \cite{Harer}, \cite{Kontsevich}, and see \cite{Penner}, \cite{Looi} and \cite{MulPenk} for related discussions). \\

Strebel differentials are not the only ones with a pole of order two, though; indeed, the local structure of the horizontal foliation depends on the coefficient of the $z^{-2}$-term. In the generic case, such a foliation comprises leaves that are not closed, but \textit{spiral} into the puncture (see the bottom of Figure 3(b))\index{measured foliation!at a pole of order two}.  A generalization of Strebel's theorem to cover such foliations was proved in \cite{GupWolf2}; see also the remark following Theorem \ref{genhm}.

\subsubsection*{Harmonic maps}\index{harmonic map} Meromorphic quadratic differentials with poles of order two also arise in the context of harmonic maps; in \cite{Wolf-Inf} Wolf showed that such a differential arises as the Hopf differential of a harmonic map of a noded Riemann surface to a hyperbolic surface obtained by ``opening up" the node. Indeed, a class of such differentials, together with a Fenchel-Nielsen\index{Fenchel-Nielsen}  twist parameter,  parametrize a  neighborhoods of a point in the boundary of augmented Teichm\"{u}ller space  $\hat{\tg}$ (that descends to the Deligne-Mumford compactification of $\mathcal{M}_g$ under the quotient by the action of the mapping class group). For some other instances where such differentials with double poles appear in this context, see \cite{Masur-Ext} or \cite{Hubb-Koch}.

\subsubsection*{An application}  As indicated in the beginning of this subsection, Strebel differentials\index{differential!Strebel} arise as limits of \Te rays\index{Teichm\"{u}ller!ray} determined by Jenkins-Strebel differentials. Thus, they are in fact useful in the study of the asymptotic behaviour of  such \Te rays -- see \cite{Amano1,Amano2}.  As we shall mention in the next subsection, higher order poles arise in limits of more general \Te rays, and we expect such meromorphic quadratic differentials  to play a role in analogous results.

\subsection{Higher order poles} 
For simplicity, we shall assume, throughout this section, that the Riemann surface $X$ has a \textit{single} puncture where the quadratic differential has a pole of order \textit{greater} than two. 

\subsubsection*{Polynomial quadratic differentials}\index{polynomial quadratic differential} The first case of interest is when $X$ is the complex plane; an elementary computation shows that a quadratic differential of the form $(z^{d} + a_{d-2}z^{d-2} + \cdots + a_0) dz^2$ has a pole of order $(d+2)$ at infinity.  The coefficients $a_0,\ldots a_{d-2}$ are complex numbers, and any such polynomial of degree $d$ can be taken to be monic and centered, as above, by composing with a suitable automorphism of $\C$.

The simplest example is the \textit{constant} differential $dz^2$ on $\C$; its induced metric is just the standard Euclidean metric on the plane, and its horizontal (resp. vertical)  foliation are the horizontal (resp. vertical)  lines on $\C$. In general,  a polynomial quadratic differential of degree $d$ induces a singular-flat metric  on $\C$ comprising $(d+2)$ Euclidean half-planes, and some number of infinite Euclidean strips. The induced foliation is the standard foliation on these domains, and  its leaf space is a planar metric tree. This tree has $(d+2)$ infinite rays dual to the foliated half-planes, and the remaining edges dual to the foliated strips,   having lengths equal to the transverse measures across them. In fact, this description allows one to show that the space of such measured foliations $\mathcal{MF}(\C, d+2)$\index{measured foliation!polynomial quadratic differential} is homeomorphic to $\R^{d-1}$ (see, for example, \cite{GupTD}). 

Such a polynomial quadratic differential also arises as the Hopf differential of a harmonic map from $\C$ to the hyperbolic plane (for the existence of such a map, see \cite{Wan}, and for a discussion of the uniqueness, the recent work \cite{Li}). The structure of the horizontal and vertical foliations discussed above, together with the estimates mentioned in Proposition \ref{gest}, implies that the image of such a map is an ideal  $(d+2)$-gon in $\mathbb{H}^2$ (see Figure 4), and in fact all  such ideal polygons arise this way (see \cite{H-T-T-W}).

 \begin{figure}
  \centering
  \includegraphics[scale=0.45]{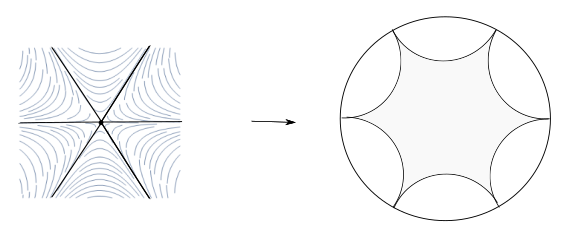}\\
    \caption{The induced horizontal foliation of $z^4dz^2$  on $\C$ (left) and the ideal polygon that is the image of the corresponding harmonic map (right). Proposition \ref{gest} implies that horizontal leaves in each sector that are far from the zero, map close to the geodesic sides. }
\end{figure}

\subsubsection*{Measured foliation with a pole singularity}\index{measured foliation!pole singularity} More generally, on a once-punctured Riemann surface $X$ of higher genus,  the local structure of the induced singular-flat geometry around any pole of order $n\geq 3$ comprises $(n-2)$ foliated half-planes.  This can also be seen in the example of the constant quadratic differential $dz^2$ on $\C$: any neighborhood $U$ of the pole at infinity would contain exactly two half-planes that are foliated by horizontal lines. Throughout this section, we shall consider the induced \textit{horizontal} foliation, with an understanding that the same results hold if considers the vertical foliation as well. 

By a result of Strebel (see Theorem 7.4 of \cite{Streb}), there is such a neighborhood $U$ of the pole such that any horizontal leaf entering it terminates at the pole. 
The induced measured foliation $F$ of a meromorphic quadratic differential on $X$ has additional real parameters at the pole, that records the structure of its restriction $F\vert_U$ . These parameters  can be thought of as lengths of edges in the metric graph that is the leaf-space of  the lift of the restricted foliation $F\vert_U$ (see Figure 5). Note that this metric graph has  $(n-2)$ infinite rays towards the pole, corresponding to each of the foliated half-planes.   In the universal cover,  the leaf-space of the lift  $\tilde{F}$ of the \textit{entire} foliation  is  then an ``augmented" $\R$-tree, which has an equivariant collection of such infinite rays. The endpoints of these rays determine an equivariant collection of additional vertices, and consequently finite-length edges, in the augmented $\mathbb{R}$-tree\index{$\mathbb{R}$-tree!augmented}. Analyzing the structure of these $\mathbb{R}$-trees, and in particular taking into account the length parameters of these additional edges,  one can show that the space of such measured foliations $\mathcal{MF}(X,n)$ is homeomorphic to $\R^{6g-6+n+1}$  (see \cite{GupWolf3}). \\

\begin{figure}
	
  \centering
  \includegraphics[scale=0.4]{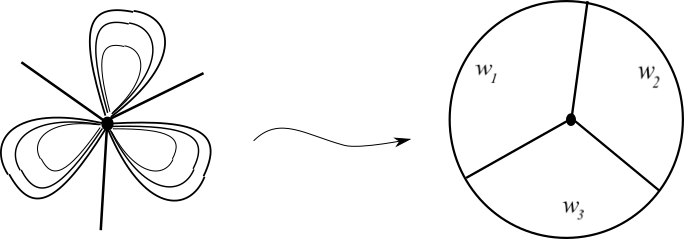}\\
    \caption{The foliation around a pole singularity (shown as the black dot) comprises foliated half-planes that appear as the ``petals" in the figure (left).  Its leaf-space is a metric graph with an infinite ray corresponding to each half-plane (right).  }
\end{figure}

The generalization of the Hubbard-Masur theorem then relies on fixing the \textit{principal part}\index{principal part} $P$ at the pole of order $n$, with respect to a coordinate disk $U\cong \mathbb{D}$ around it.  This is defined to be the polar part of $\sqrt{q\vert_U}$ comprising the terms for $z^{-i}$ where $i\geq 2$. Note that the difference $\sqrt{q\vert_U} - P$ has at most a simple pole at  the origin. In order to make sense of $\sqrt{q\vert_U}$ when $n$ is odd, the principal part $P$ can be thought of as a meromorphic $1$-form defined on the branched double cover of $U$.\\

\textit{Example.} We illustrate the last sentence with the following computation.  Consider the polynomial quadratic differential $q= (z^3 + az+b)dz^2$ where $a,b\in \C$.  This quadratic differential has a  pole of order $7$ at $\infty$, where it is of the form  $(w^{-7} + aw^{-5} + bw^{-4})dw^2$ on a disk $U$ equipped with the coordinate $w$ obtained by inversion. Substituting $\eta^2 = w$ to pass to the branched double cover $\hat{U}$, and taking a square root, we obtain
$\sqrt{q\vert_{\hat{U}}} =  2\left(\eta^{-6} + {\frac{1}{2}} a\eta^{-2} + \cdots\right) d\eta$ and thus the principal part $P$ of the original quadratic differential $q$ involves the complex coefficient $a$, and not $b$. \\

When the order of the pole $n$ is even, the real part of the {residue} of the principal part $P$\index{principal part!residue} is determined by the transverse measures of the induced horizontal foliation around the pole (see Lemma 5 of \cite{GupWolf3} for details). If the residue satisfies that, we say that the principal part $P$ is \textit{compatible} with the foliation; we say this compatibility always holds in the case that $n$ is odd.

We then have:

\begin{thm}[\cite{GupWolf3}]\label{genhm}\index{Hubbard-Masur theorem!generalization} Let $X$ be a once-punctured Riemann surface of genus $g\geq 1$ and fix a coordinate disk $U$ around the puncture $p$. Let $n\geq 3$ and let $F\in \mathcal{MF}(X,n)$ be a measured foliation on $X$ with a pole singularity at $p$. Then for any choice of principal part $P$ that is compatible with the foliation $F$, there exists a unique meromorphic quadratic differential with a pole of order $n$ at $p$ with that principal part on $U$, whose horizontal foliation is $F$. 
\end{thm}

\textit{Remarks.} (1) Just like Theorem \ref{hm}, the above holds for the \textit{vertical} foliation as well; the compatibility condition then involves the imaginary part of the residue.  Moreover, the assumption on the genus $g$ is an artifact of the fact that we have assumed there is exactly one puncture; more generally, we only need the Euler characteristic $\chi(X)<0$.

(2) The same statement as above holds when $n=2$ (see \cite{GupWolf2}); in this case, the residue at the pole is determined by the coefficient of $z^{-2}$ at the order two pole, and determines the ``spiralling" nature of the horizontal foliation.\\

A special case of the above theorem is when the foliation $F$ has the property that all the non-critical leaves lie in foliated half-planes. The critical graph then forms a metric spine for the punctured surface $X \setminus p$, and a meromorphic quadratic differential with such a horizontal foliation is the analogue of Strebel differentials that we saw earlier.  The induced singular-flat metric is what we call a \textit{half-plane structure}\index{half-plane structure} comprising Euclidean half-planes attached to the critical graph. The choice of the principal part $P$ determines the singular-flat geometry of the resulting ``planar end" around the puncture (see \cite{GupWolf1}).

Like for Strebel differentials, such half-plane structures arise as limits of singular-flat surfaces along a \Te rays\index{Teichm\"{u}ller!ray}; see \S2 of \cite{GupLimits} for details, and \cite{Gup2} or \cite{GupLimits} for some applications.\\

A \textit{generic} meromorphic quadratic differential induces a  (horizontal or vertical)  measured foliation such that each leaf has at least one end terminating at the pole; as for polynomial quadratic differentials on $\C$, the induced singular-flat metric then comprises half-planes and infinite strips.  This trajectory structure\index{measured foliation!trajectory structure} is related to the theory of wall-crossing, Donaldson-Thomas invariants and stability conditions in the sense of Bridgeland (see \cite{BriSmi} and \cite{HKK}). 
For further work investigating the trajectory structure on various strata, that can be defined for the total space of the bundle of  meromorphic quadratic differentials over $\mathcal{T}_{g,1}$, see \cite{Tahar}.  For an investigation of $\mathrm{SL}_2(\R)$-dynamics on such strata, see \cite{Boissy}.

\subsubsection*{Crowned hyperbolic surfaces}\index{crowned hyperbolic surface} 
The generalization of Wolf's parametrization (Theorem \ref{wparam})  to the case of meromorphic quadratic differentials involves considering hyperbolic surfaces with certain non-compact ends corresponding to the higher order poles. Such an end is a \textit{hyperbolic crown} bordered by a cyclic collection of bi-infinite geodesics, each adjacent pair of which encloses a ``boundary cusp".  See Figure 6.
 The space of such crowned surfaces of genus $g\geq 1$, such that the crown end has $m\geq 1$ boundary cusps, is the ``wild" \Te space  $\mathcal{T}_g(m)$\index{Teichm\"{u}ller space!wild} that can be shown to be homeomorphic to $\mathbb{R}^{6g-6 + m+3}$ (see \S3 of \cite{GupWild}).  
 
 \begin{figure}
  \centering
  \includegraphics[scale=0.35]{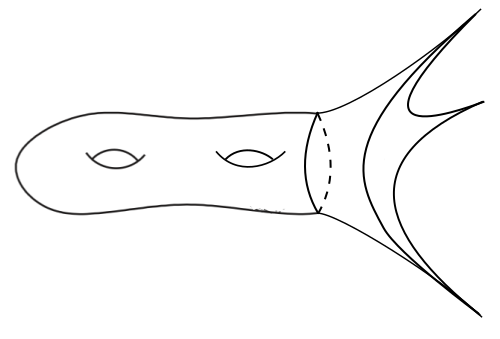}\\
    \caption{A crowned hyperbolic surface.}
\end{figure}

When $m$ is even, the \textit{metric residue} of the crown end is calculated by considering a truncation of the cusps, and taking an alternating sum of the lengths of the resulting geodesic sides (this is independent of the choice of truncation). \\

In \cite{GupWild} we proved:\index{Wolf's parametrization!generalization} 
\begin{thm}\label{wild} Let $X$ be a once-punctured Riemann surface of genus $g\geq 1$ and fix a coordinate disk $U$ around  the puncture $p$. Let 
\begin{itemize}
\item $Q(X, p, P)$ be the space of meromorphic quadratic differentials with a pole of order $n$ at $p$ that has principal part $P$ on $U$, 

\item $\mathcal{T}_g(P)$ be the subspace of  $\mathcal{T}_g(m)$ such that the metric residue of the crown end equals twice the real part of the residue of $P$\index{principal part!residue}, if $m$ is even and

\item $\mathcal{T}_g(P)$ be equal to  $\mathcal{T}_g(m)$ if $m$ is odd. 

\end{itemize}

Then  for any $Y \in \mathcal{T}_g(P)$, there is a unique harmonic map $h:X\setminus p \to Y$ such that the Hopf differential $\Psi(Y)$ lies in $Q(X, p, P)$. This defines a homeomorphism $\Psi: \mathcal{T}_g(P) \to Q(X, p, P)$.

\end{thm}

\vspace{.1in}

\textit{Remarks.} (1) The constraint imposed by the residue in the case $m$ is even is a consequence of Proposition \ref{gest}, since horizontal leaves in the foliated half-planes around the pole map close to the geodesic sides of the crown end, with twice the length. \\
(2)  The theorem above holds in the case that $X = \C$ as well; in this case $Q(X, p, P)$ is the space of polynomial quadratic differentials\index{polynomial quadratic differential} of degree $(n-2)$ that we saw earlier, and $\mathcal{T}_g(P)$ is the space of ideal $n$-gons.  In the case of $n=4$ (degree two polynomials) the correspondence is explicit, since the metric residue of an ideal quadrilateral determines its cross ratio; this correspondence was observed in a different way in \cite{Au-Wan-Symm}.

\subsubsection*{Strategy of the proof} 
In Theorem \ref{wild}, even the \textit{existence}  of such a harmonic map from a non-compact domain\index{harmonic map!from a punctured Riemann surface} to a crowned hyperbolic surface requires work, the additional difficulty being that such a map  necessarily has \textit{infinite} energy  since the Hopf differential has a higher order pole.\\

This existence proof  involves the following steps:

\begin{itemize}

\item  \textit{Step 1.}  Determine a space of \textit{model harmonic maps} defined on a punctured disk. A crucial observation here is that the principal part of the Hopf differential determines the asymptotic behavior at the puncture. In particular, our desired principal part $P$ picks out a model map $\mu$ that our final harmonic map should be asymptotic to, on the coordinate disk $U$. \\

\item  \textit{Step 2.}  Define a sequence of harmonic maps $\{h_i\}_{i\geq 0}$ defined on a compact exhaustion  $$X_0 \subset X_1 \subset \cdots \subset X_i \subset \cdots$$
of the punctured surface $X\setminus p$, such that each solves a Dirichlet problem with the boundary condition determined by the desired model map $\mu$.  The key analytic work is to then show that  there is a uniform bound on the energy of these maps when restricted to a fixed compact subsurface.  This relies on the analysis of a \textit{partial boundary value problem}  (see \cite{AndyThesis}) on an annular region that is the difference $X_{i+1} \setminus X_0$, and proving that it is a uniformly bounded distance from the model map $\mu$. \\

\item \textit{Step 3.}  Finally, the convergent subsequence of harmonic maps obtained in the preceding step is shown to converge to a harmonic map on the punctured Riemann surface that has principal part $P$. \\

\end{itemize}

\textit{Remark.}   The same strategy is used in the proof of Theorem \ref{genhm}; in that case the model harmonic maps have a target metric tree that is dual to the lift of the desired foliation on $U$ to the universal cover.    \\

An alternative approach for the proof of the existence, using the method of sub- and super-solutions, is implicit in the work of \cite{Nie} (see also \S2 of \cite{Wild-Tambu}). In the more general context of the non-abelian Hodge correspondence\index{non-abelian Hodge correspondence} mentioned in \S5, the case when the Higgs field has \textit{irregular} or higher order poles was considered by Biquard and Boalch in \cite{BiqB}, where they also proved an existence theorem for the corresponding harmonic maps. \\

In \cite{GupLimits} we show how harmonic maps from a punctured Riemann surface to a crowned hyperbolic surface also arise in limits of harmonic maps between compact surfaces (as in Wolf's parametrization), when the target hyperbolic surface is fixed, and the  {domain} Riemann surfaces \textit{diverges} along a \Te ray\index{Teichm\"{u}ller!ray} which has a half-plane structure as a limit.

\subsection{Meromorphic projective structures} Finally, we mention some ongoing work generalizing the relation of quadratic differentials with projective structures that we saw in \S7\index{projective structure!meromorphic}.  

Meromorphic quadratic differentials with poles of order two arise as Schwarzian derivatives\index{Schwarzian derivative} of \textit{branched} projective structures (see, for example, \cite{Luo} or \S1.4 of \cite{GKM} ).  In the case of poles of higher order, the corresponding solutions of the Schwarzian equation define a space of \textit{meromorphic projective structures}, and the holonomy map defined on the space of such structures was studied in the recent work of \cite{AllBrid}.  The Schwarzian equation is part of a broader  study of linear differential systems on the complex plane, and in this context, the case of polynomial quadratic differentials on $\C$\index{polynomial quadratic differential}  is part of the classical literature (see, for example, \cite{Sib-book}). 

 It would be interesting to study the geometry of projective structures on punctured Riemann surfaces.  Meromorphic quadratic differentials are also known to arise in compactifications of various strata of the total space $Q_g$ of quadratic differentials -- see \cite{BCGGM1,BCGGM2}, \cite{Gendron} or Theorem 10 of \cite{EKZ}. It would also be interesting to explore the resulting compactification of the space of projective structures $\mathcal{P}_g$, and the degenerations that lead to the meromorphic structures on the boundary.

\vspace{.5in}

\bibliographystyle{amsplain}
\bibliography{hdrefs}

\end{document}